\documentclass{amsart}
\usepackage{amssymb}
\usepackage{amscd}
\usepackage{amsmath}
\usepackage{enumerate}
\usepackage[all]{xy}
\usepackage[pdftex]{graphicx}
\newtheorem{theo}{Theorem}[section]
\newtheorem{lemma}[theo]{Lemma}
\newtheorem{cor}[theo]{Corollary}
\newtheorem{prop}[theo]{Proposition}

\newtheorem{defi}[theo]{Definition}

\newtheorem{remark}[theo]{Remark}

\newcommand{\bC}{{\mathbb{C}}}
\newcommand{\bN}{{\mathbb{N}}}

\newcommand{\bR}{{\mathbb{R}}}

\newcommand{\bS}{{\mathbb{S}}}

\def\e{{\epsilon}}
\def\g{{\gamma}}

\def\a{{\alpha}}
\def\b{{\beta}}
\newcommand{\fgbar}{f\bar{g}}
\newcommand{\B}{\textbf{B}}
\newcommand{\BS}{\textbf{S}}
\newcommand{\BD}{\textbf{D}}
\newcommand{\s}{\vspace{0,3cm}}

\begin{document}
\title[THE BOUNDARY OF THE MILNOR FIBRE]{THE DEGENERATION OF THE BOUNDARY OF THE MILNOR FIBRE TO THE LINK OF COMPLEX AND REAL NON-ISOLATED SINGULARITIES}
\author{Aur\'elio Menegon Neto}
\author{Jos\'e Seade}

\thanks{Partial support from CONACYT and UNAM (Mexico), and FAPESP (Brazil).}
\address{Aur\'elio Menegon Neto: Instituto de Ci\^encias Matem\'aticas e de Computa\c c\~ao, Universidade de S\~ao Paulo, S\~ao Carlos - Brazil.}
\address{Jos\'e Seade: Instituto de Matem\'aticas, Unidad Cuernavaca, Universidad Nacional Aut\'onoma de M\'exico}
\email{aurelio@icmc.usp.br}
\email{jseade@matcuer.unam.mx}

\begin{abstract}
We study the boundary of the Milnor fibre of real analytic singularities $f: (\bR^m,0) \to (\bR^k,0)$, $m\geq k$, with an isolated critical value and the Thom $a_f$-property. We define the vanishing zone for $f$ and we give necessary and sufficient conditions for it to be a fibre bundle over the link of the singular set of $f^{-1}(0)$. In the case of  singularities  of the type $\fgbar: (\bC^n,0) \to (\bC,0)$ with an isoalted critical value, $f, g$ holomorphic, we further describe the degeneration of the boundary of the Milnor fibre to the link of $\fgbar$.
As a milestone, we also construct a L\^e's polyhedron for real analytic singularities of the type $\fgbar: (\bC^2,0) \to (\bC,0)$ such that either $f$ or $g$ depends only on one variable.

\end{abstract}
\maketitle

\section*{Introduction}

A fundamental problem in mathematics is understanding the way how the non-critical levels $F_t := f^{-1}(t) \cap \B_\e$ of an analytic map-germ
$(\bR^m,0) \buildrel{f} \over {\to} (\bR^k,0)$, $m \ge k$, degenerate to the special fibre $F_0:= f^{-1}(0) \cap \B_\e$. For instance, when $f$ is holomorphic, the celebrated Milnor's fibration theorem says that the  $F_t $ form a locally trivial fibration over a punctured disc. And a  remarkable theorem of  L\^e D\~ung Tr\'ang \cite{Le1} says that when $f$ further has an isolated critical point, inside each  $F_t$ one has a polyhedron $P_t$ of middle dimension which ``collapses" as we approach the special fiber, and the complement $F_t \setminus P_t$ is diffeomorphic to $F_0 \setminus \{0\}$.  If $f$ is holomorphic with non-isolated singularities, one still has Milnor's fibration theorem and there is a rich theory, developed by D. Massey and many others, explaining how the fibers $F_t $ degenerate into the special fiber $F_0$.

There is an alternative viewpoint to this problem, which arises from work by Siersma \cite{Si}, Michel-Pichon \cite{MP, MP2} and Nemethi-Szilard \cite{NS} (see also \cite {FM}). For this we recall that given a real analytic map-germ $f$ as above one has the link of $f$, that we denote by $L_0$, which by definition is $L_0:= f^{-1}(0) \cap \bS_\e$,  the intersection with a sufficiently small sphere.  The link can be regarded as being the boundary of $F_0$, and it  is non-singular if and only if $f$ has no other critical points near $0$. In that case $L_0$ is isotopic to the boundary $L_t$ of the Milnor fiber $F_t$, and  this crucial fact was essential in the classical work of Milnor, Brieskorn, Hirzebruck and others, to study the topology of the link $L_0$.

In general the link $L_0$ is a real analytic singular variety, whose homeomorphism type is independent of all choices and determines the topology of $V_0$. The boundary of the Milnor fibre is a smooth manifold with a rich topology (see for instance \cite {MP, FM, NS}), and these manifolds $L_t$ degenerate into the link $L_0$. In this article we study this degeneration  $L_t \to L_0$.
 
We recall that the links of isolated singularities in complex analytic spaces have proved to be a rich family of manifolds to be considered for various problems in geometry and topology. The setting we envisage here provides a larger class of manifolds, which still come equipped with a rich geometry and topology.

We first prove (Theorem \ref {theo_c4}):

\vskip.2cm

\noindent
{\bf Theorem.} {\it 
Let $f: (\bR^m,0) \to (\bR^k,0)$, with $m\geq k$, be a real analytic map-germ such that $0 \in \bR^k$ is an isolated critical value and $f$ has the Thom $a_f$-property. Let $\Sigma$ be the singular set of $V:= f^{-1}(0) \cap \B_\e$ and let $L(\Sigma):= \Sigma \cap \bS_\e$ be its link. Then:

\vskip.1cm
\noindent {\bf $(a)$}  There exists a compact  regular neighbourhood $W$ of $L(\Sigma)$ in $\BS_\e$ with smooth boundary, such that for all $t$ with $\Vert t \Vert$ sufficiently small, one has that $L_t \backslash \mathring{W}$ is homeomorphic to $L_0 \backslash \mathring{W}$, where $ \mathring{W}$ is the interior of $W$. 

\vskip.1cm

\noindent {\bf $(b)$} If we suppose that $\Sigma$ has at most an isolated singularity,  then $W$  can be chosen to be a fibre bundle over $L(\Sigma)$ with fibre a disk.  

\vskip.1cm

\noindent {\bf $(c)$}  In this latter case one has that the intersection $W_t:= L_t \cap W_t$ is a fibre bundle over $L(\Sigma)$ if and only if the intersection $W_0:= L_0 \cap W_t$ is a fibre bundle over $L(\Sigma)$, and this happens if and only if  $\Sigma \backslash \{0\}$ is a stratum of a Whitney stratification of $f$.}

\vskip.2cm

We remark that if $f$ is as above, then it has a Milnor-L\^e fibration (see for instance \cite {PS}):
$$f_|: (f)^{-1}(\BD_\eta^*) \cap \B_\e \to \BD_\eta^*\,,$$
where $\B_\e$ denotes the closed ball in $\bR^m$ centered at $0$ and with radius $\e$, $\BD_\eta$ denotes the closed ball in $\bR^k$ centered at $0$ and with radius $\eta$ and $\BD_\eta^*$ denotes the punctured disk $\BD_\eta \backslash \{0\}$.

Following \cite{Si, MP} we call a manifold $W$ as in the theorem above a {\it vanishing zone} for $f$.  So the idea is that outside the vanishing zone $W$, the link $L_0$ is smooth and diffeomorphic to that part of the boundary of the Milnor fibre which lies outside $W$. Then one must focus on what happens inside $W$. This is interesting for two reasons. On the one hand the boundary of the Milnor manifold, being a smooth manifold, can be easier to handle than the link, and understanding the way how $L_t$ degenerates into $L_0$ throws light into the topology of the link, just as the study of the vanishing cycles on the Milnor fiber throws light into the topology of the special fiber. On the other hand we can argue conversely, and understanding the degeneration $L_t \to L_0$ allows us to re-construct  $L_t$ out from $V:= f^{-1}(0) \cap \B_\e$ itself. For instance, this was the approach followed in \cite{MP, NS, FM} to show that in the cases under consideration in those articles, the boundary $L_t$ is a Waldhausen manifold.

In the second part of this article we concentrate on analytic functions of the form: $$\fgbar: (\bC^n,0) \to (\bC,0)\,,$$ 
where  $f,g: (\bC^n,0) \to (\bC,0)$ are  holomorphic  functions such that the real analytic map-germ $\fgbar$ 
has an isolated critical value at $0 \in \bC$. In \cite{PS2} Pichon and Seade proved that such a map-germ has the Thom $a_f$-property and hence there exist sufficient small positive reals $0 < \eta << \e$ such that the restriction 
$$\fgbar_|: (\fgbar)^{-1}(\BD_\eta^*) \cap \B_\e \to \BD_\eta^*$$
is a locally trivial fibration. 

For this we first extend  (in  section \ref{section_LP-real}) L\^e's  construction in \cite{Le1} of a ``vanishing polyhedron", to the case of map-germs of the type $\fgbar: (\bC^2,0) \to (\bC,0)$ with an isolated critical point. And  then, using the results of sections \ref{section_BMF} and \ref{section_LP-real}, we  study in section \ref{section_degeneration}  the degeneration of the boundary of the Milnor fibre to the link, for real analytic map-germs of the type $\fgbar: (\bC^n,0) \to (\bC,0)$ satisfying some hypothesis.

Notice that if $g$ is constant, then $\fgbar$ is a holomorphic map-germ.
In \cite{MP},  F. Michel and A. Pichon describe very precisely a way of constructing a vanishing zone and the boundary of the Milnor fibre for holomorphic map-germs $(\bC^3,0) \to (\bC,0)$. This was done also in \cite{FM} in a different way, somehow inspired by Nemethi-Szilard's work \cite{NS}, and that construction works also for map germs $\fgbar: (\bC^3,0) \to (\bC,0)$. The constructions  we give in this article do not give so much information about the boundary of the Milnor as those in \cite{Si, MP, FM}, but they have the advantage of applying in  very general settings, thus throwing light on the topology of real analytic map-germs with non-isolated singularities.

The authors are grateful to M. A. S. Ruas for helpful discussions.

\vskip.5cm
\section{The boundary of the Milnor fibre of non-isolated singularities}
\label{section_BMF}

Consider now  a real analytic map germ 
$$f: (\bR^m,0) \to (\bR^k,0)\;, \quad m \geq k\,,$$
 such that $0 \in \bR^k$ is an isolated critical value and $f$ has the Thom $a_f$-property, so $f$ has a local Milnor-L\^e fibration. Set $V:= f^{-1}(0)$ and let 
$L_0:= V \cap \BS_\e$ be 
the link of  $f$. We denote by $L(\Sigma) := \Sigma \cap \BS_\e$
the link of the singular set of $V$. By $L_t$ we denote the boundary of a Milnor fibre $F_t$, so $L_t := f^{-1}(t) \cap \BS_\e$ for $t$ such that $0 \ll |t| \ll \e$. Then $L_t$ is a smooth submanifold of the sphere $\BS_\e$ and as $ |t|$ tends to $0$, this manifold degenerates to the link $L_0$. 
Following \cite{Si} and \cite{MP} (also \cite {FM}), in this section we look at the way how $L_t$ degenerates into $L_0$.

\subsection{The vanishing zone for real analytic non-isolated singularities} $\,$
\medskip

We have: 

\begin{defi}
A {\it vanishing zone} for $L_0$ is a regular neighborhood  $W$ of $L(\Sigma)$ in $\BS_\e$ with smooth boundary, such that for all $t$ sufficiently near $0$, the smooth manifold $L_t \backslash \mathring{W}$ is diffeomorphic to $L_0 \backslash \mathring{W}$.
\end{defi}

We shall construct a vanishing zone for every real analytic map-germ $f$ as above, using the idea of a {\it cellular tube} of a submanifold $S$ of a manifold $M$, defined by Brasselet in \cite{Br}. This construction  is classical in PL-topology and generalizes the concept of  tubular neighborhoods. 

\begin{lemma} \label{vz}
There exists a compact regular neighbourhood $W$ of $L(\Sigma)$ in $\BS_\e$ such that its boundary $\partial W$ is smooth and it intersects $L_0$ transversally. Moreover, if $\Sigma$ is either a smooth manifold or it has an isolated singularity, then $W$ is a fibre bundle over $L(\Sigma)$ with fibre an $(m-r)$-dimensional disk in $\BS_\e$, where $r$ is the dimension of $\Sigma$.
\end{lemma}

\begin{proof}
Let us consider a Whitney stratification of $\BS_\e$ such that $L_0$ is a union of strata. Now let us consider a triangulation $(K)$ of $\BS_\e$ such that each strata of the Whitney stratification is a union of simplices. Let $(K')$ be the barycentric decomposition of $(K)$. 
Using $(K')$ one constructs the associated cellular dual decomposition $(D)$ of $\BS_\e$: given a simplex $\sigma$ in $(K)$ of dimension
$s$, its dual $d(\sigma)$ is the union of all simplices $\tau$ in
$(K')$ whose closure meets $\sigma$ exactly at its barycenter
$\hat{\sigma}$, that is, 
$$\bar{\tau} \cap \sigma = \hat{\sigma}.$$
It is a cell of dimension $(m-s-1)$. Taking the union of all these dual cells we get the dual decomposition $(D)$ of $(K)$. By construction, each cell $\sigma$ intersects its dual $d(\sigma)$ transversally. 
We let $W$ be the union of cells in $(D)$ which are dual of simplices in $L(\Sigma)$; it provides a cellular tube around $L(\Sigma)$ in $\BS_\e$, which means it satisfies the following properties:

\begin{itemize}
\item[$(i)$] $W$ is a compact neighbourhood of $L(\Sigma)$ containing $L(\Sigma)$ in its interior, and $\partial W$ is a retract of $W \backslash L(\Sigma)$;
\item[$(ii)$] $W$ retracts to $L(\Sigma)$;
\item[$(iii)$] For any neighbourhood $U$ of $L(\Sigma)$ in $\BS_\e$, if the triangulation of $\BS_\e$ is sufficiently ``fine" then we can assume $W \subset U$.
\end{itemize}
Now consider the sets:
\begin{itemize}
\item[$\bullet$] $A := \{ \sigma \in (K) \, \big | \ \sigma$ is a $(1)$-simplex of $L_0$ whose closure intersects $\partial W \}$ \,;
\item[$\bullet$] $B := \{ \sigma \in (K')  \, \big |  \sigma$ is a $(m-2)$ simplex in $\partial W$ whose closure intersects $L_0 \} \,$.
\end{itemize}
Then one can see that 
$$B = \bigcup_{\sigma \in A} d(\sigma).$$
Therefore $\partial W$ intersects $L_0$ transversally. Moreover, if $\Sigma$ is either a smooth manifold or an isolated singularity, it follows that $L(\Sigma)$ is a smooth manifold without boundary, and then it follows from section 1.1.2 of \cite{BSS} that $W$ is a bundle on $L(\Sigma)$ whose fibres are disks. 
Finally, by a theorem of Hirsch (\cite{Hs}) we can assume that $\partial W$ is a smooth manifold.
\end{proof}

\begin{cor} \label{cor_1}
For $t$ sufficiently small, $L_t$ intersects $\partial W$ transversally. As a consequence, $\partial W_t := \partial W \cap L_t$ is a differentiable manifold.
\end{cor}

\begin{proof} 
Suppose this is not true, that is, suppose that there exists a sequence of points $(p_t)$ in $f^{-1}(\B_\eta^*) \cap \partial W$, with $p_t \in L_t \cap \partial W$, which converges to $p_0 \in L_0 \cap \partial W$, such that $T_{p_t} L_t$ intersects $T_{p_t} \partial W$ not transversally, for each $p_t$. Set 
$$T := \lim_{t \to 0} \ T_{p_t} L_t$$
and let $Reg(L_0)$ denote the regular part of $L_0$. Since $f$ has the Thom $a_f$-property, it follows that $T_{p_0} (Reg L_0) \subset T$, and since $T_{p_0}(Reg L_0) \pitchfork T_{p_0} \partial W$ by the previous lemma, it follows that $T \pitchfork T_{p_0} \partial W$, {\it i.e.}, $T$ and $T_{p_0}$ meet transversally. Consider $d$ a metric in the corresponding  Grassmannian. Since transversality is an open property,  we have that if $d(T, T_{p_t} L_t)$ and $d(T_{p_0} \partial W, T_{p_t} \partial W)$ are sufficiently small, then $T_{p_t} L_t \pitchfork T_{p_t} \partial W$, which is a contradiction. 
Therefore, if $\eta$ is sufficiently small, one has that $T_{p_t}L_t \pitchfork T_{p_t} \partial W$, for any $t \in \BD_\eta$ and $p_t \in (L_t \cap \partial W)$.
\end{proof}

\begin{lemma} \label{lemma_vz}
For any $t \in \B_\eta$, the smooth manifold $L_t \backslash \mathring{W}$ is diffeomorphic to $L_0 \backslash \mathring{W}$.
\end{lemma}

\begin{proof}
Set $M = \BS_\e \cap f^{-1}(\B_\eta) \backslash \mathring{W}$. By Ehresmann's fibration lemma \index{Ehresmann's fibration lemma} for manifolds with boundary, $f_{|_M}: \BS_\e \cap f^{-1}(\B_\eta) \backslash \mathring{W} \to \B_\eta$ is a fibre bundle
if one has the following conditions:
\begin{itemize}
\item[(1)] $\forall p \in M$, $D(f_{|_M})_p: T_p(\BS_\e)
\to T_{f(p)}\bC$ is a surjection;
\item[(2)] $\forall p \in \partial M$, $D(f_{|_{\partial M}})_p:
T_p(\partial M) \to T_{f(p)}\bC$ is a surjection.
\end{itemize}
These conditions are equivalent to the following:
\begin{itemize}
\item[(1')] the fibres of $f$ are transversal to $\BS_\e$;
\item[(2')] the fibres of $f$ are transversal to $\partial M$.
\end{itemize}
But $(1')$ follows from Milnor theory (\cite{Mi}) and $(2')$ follows from Corollary \ref{cor_1}. 
\end{proof}

\subsection{The topology of the vanishing zone}
\label{section_TVZ}

From now on we suppose that $\Sigma$ is either a smooth manifold or it has an isolated singularity, of dimension $r$, with $0<r\leq m-2$. We want to describe the topology of $L_0$ inside $W_0 := W \cap L_0$, and the topology of $L_t$ inside $W_t := W \cap L_t$. By construction, we know that $W$ is a fibre bundle over $L(\Sigma)$, with projection $p$ and fibre the $(m-r)$-dimensional ball $\B(s)$, centered at $s \in L(\Sigma)$.

$$\xymatrix{ 
\B(s) \ \ar@{^{(}->}[r] & W \ar[d]^{p} \\ 
	&	L(\Sigma) \;.
} 
$$
For each $s \in L(\Sigma)$, consider the restriction
$$f_s: \B(s) \to \bR^m.$$
Clearly $f_s^{-1}(0)$ is a compact real analytic variety in $\B(s)$, for any $s \in L(\Sigma)$. Moreover, by Lemma \ref{vz}, it has an isolated singularity at $0 \in \B(s)$. Also, for any $t \in \B_\eta^*$ one has that $f_s^{-1}(t)$ is a smooth manifold of dimension $m-r-1$, since $t$ is a regular value of $f_s$.

\begin{lemma} \label{lemma_Wt}
There exists $\delta$ sufficiently small such that if $W$ is taken to be a fibre bundle with fibre the ball $\B_\delta(s)$ centered at $s \in L(\Sigma)$, then $\mathring{W_t}:= \mathring{W} \cap L_t$ intersects $\B_\delta(s)$ transversally, for any $s \in L(\Sigma)$ and $t \in \B_\eta$.
\end{lemma}

\begin{proof} Assume first that  $t=0$. 
We know that $L_0 \supseteq L(\Sigma)$ and that $L(\Sigma)$ intersects $\B_\delta(s)$ transversally for any $s \in L(\Sigma)$, that is, 
$$T_s(L(\Sigma)) + T_s(\B_\delta(s)) = T_s(\BS_\e).$$
We want to show that $T_y(\mathring{W_0}) + T_y(\B_\delta(s)) = T_y(\BS_\e)$ for any $y \in [\mathring{W_0} \cap \B_\delta(s)]$. Let $(S_\a)$ be a Whitney stratification of $W_0$ induced by a Whitney stratification of $f$. There are two possibilities:

{\bf Case I:} Suppose that $s$ belongs to a stratum of dimension $r-1$. Then by the $(a)$-condition of the Whitney stratification, we have that for $y$ sufficiently close to $s$ the tangent space $T_y(Reg L_0)$ is very close to a plane $P$ that contains $T_s(L(\Sigma))$ and so it is transversal to $\B_\delta(s)$. Since transversality is an open property, it follows that $\mathring{W_0}$ intersects $\B_\delta(s)$ in $y$ transversally. So if $\delta$ is sufficiently small, we have that $\mathring{W_0}$ intersects $\B_\delta(s)$ transversally.

{\bf Case II:} Suppose that $s$ belongs to a stratum of dimension $< r-1$. We can consider $s' \in L(\Sigma)$ as close to $s$ as we want such that $s'$ belongs to a stratum of dimension $r-1$. Then for $y \in \B_\delta(s)$ sufficiently close to $s$, the tangent $T_y(Reg L_0)$ is very close to a plane $P$ that contains $T_{s'}(L(\Sigma))$ and so it is transversal to $\B_\delta(s')$. But since $\B_\delta(s)$ is very close to $\B_\delta(s')$, it follows that $P$ is transversal to $\B_\delta(s)$. Therefore $\mathring{W_0}$ intersects $\B_\delta(s)$ in $y$ transversally. So if $\delta$ is sufficiently small, we have that $\mathring{W_0}$ intersects $\B_\delta(s)$ transversally.

If we take $\delta$ sufficiently small such that $\mathring{W_0}$ intersects $\B_\delta(s)$ transversally, then clearly $\mathring{W_0}$ intersects $\B_\delta(s')$ transversally, for any $s' \in L(\Sigma)$ sufficiently close to $s$. Since $L(\Sigma)$ is compact, we are done.

\medskip

Now consider the case $t \neq 0$ and let $\delta$ be as above. We know that 
$$T_{y'}(\mathring{W_0}) + T_{y'}(\B_\delta(s)) = T_{y'}(\BS_\e),$$
for any $y' \in [\mathring{W_0} \cap \B_\delta(s)]$. We want to show that 
$$T_{y}(\mathring{W_t}) + T_{y}(\B_\delta(s)) = T_{y}(\BS_\e),$$
for any $y \in [\mathring{W_t} \cap \B_\delta(s)]$. Since $f$ has the Thom $a_f$-property, we know that for $t>0$ sufficiently small, $T_y(\mathring{W_t})$ is very close to a plane $P$ that contains $T_{y'}(Reg(L_0))$, for some $y' \in Reg(L_0) \cap \B_\delta(s)$ very close to $y$. But then $P$ intersects $T_{y'}(\B_\delta(s))$ transversally and therefore $T_y(\mathring{W_t})$ intersects $T_y(\B_\delta(s))$ transversally.
\end{proof}

\begin{prop} \label{equiv}
The following conditions are equivalent:

\begin{itemize}
\item[$(i)$] $W_0$ is a fibre bundle over $L(\Sigma)$ with projection $p_0 = p_|: W_0 \to L(\Sigma)$ and fibre $f_s^{-1}(0)$;
\item[$(ii)$] $\partial W_0 = \partial W \cap L_0$ is a fibre bundle over $L(\Sigma)$ with projection $p_|: \partial W_0 \to L(\Sigma)$ and fibre $\partial \big( f_s^{-1}(0) \big)$, the link of $f_s$;
\item[$(iii)$] $\partial W_t = \partial W \cap L_t$ is a fibre bundle over $L(\Sigma)$ with projection $p_|: \partial W_t \to L(\Sigma)$ and fibre $\partial \big( f_s^{-1}(t) \big)$, the boundary of the Milnor fibre of $f_s$;
\item[$(iv)$] $W_t$ is a fibre bundle over $L(\Sigma)$ with projection $p_t = p_|: W_t \to L(\Sigma)$ and fibre $f_s^{-1}(t)$.
\end{itemize}
\end{prop}

\begin{proof} 
The implications $(i) \Rightarrow (ii)$ and $(iv) \Rightarrow (iii)$ are immediate, and  the equivalence $(ii) \Leftrightarrow (iii)$  follows from Lemma \ref{lemma_vz}. Let us prove now that $(ii) \Rightarrow (i)$. First we show that $f_s^{-1}(0)$ is homeomorphic to $f_{s'}^{-1}(0)$ for all $s,s' \in L(\Sigma)$. Fix $s \in L(\Sigma)$ and let the ball $\B_\theta(s) \subset \B_\delta(s)$ centered at $s$ be a Milnor ball for $f_s$.

We claim that there exists a neighborhood $V_s$ of $s$ in $L(\Sigma)$ such that for any $s' \in V_s$, one has that $\B_\theta(s') \subset\B_\delta(s')$ is a Milnor ball for $f_{s'}$. To prove this claim, suppose this is not true, {\it i.e.},  there exist $s'$ as close to $s$ as one wishes, and $\theta'$ with $0< \theta' \leq \theta$ such that $f_{s'}^{-1}(0)$ intersects $\BS_{\theta'}$ non transversally. Set $\delta = \theta'$. Since $s'$ is close to $s$ and since $f$ is continuous, if $f_s^{-1}(0)$ intersects (transversally) $\BS_{\theta'}$ in $q$ connected components, then $f_{s'}^{-1}(0)$ intersects $\BS_{\theta'}$ transversally in at least $q$ connected components, and non transversally in at least one. Then $\partial W_0$ is not a fibre bundle over $L(\Sigma)$, which a contradiction (see figure \ref{fig3}).

\begin{figure}[!h] 
\centering 
\includegraphics[scale=0.5]{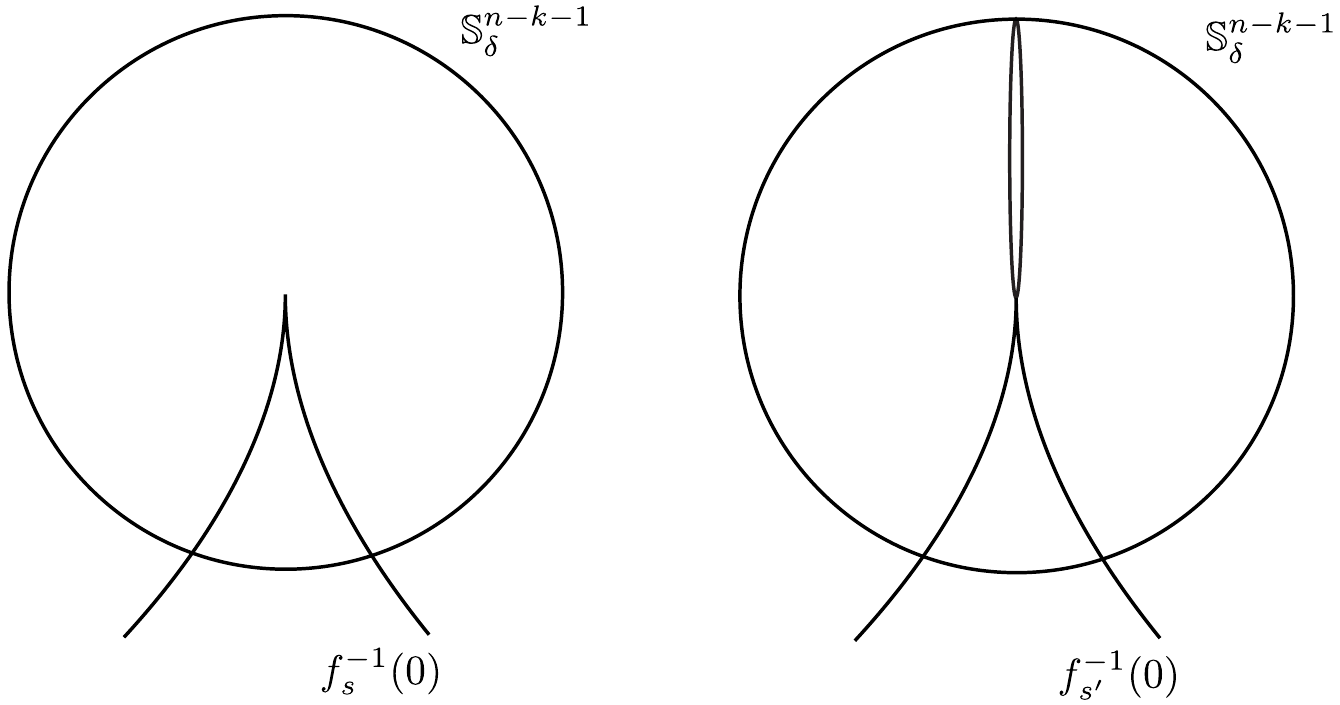}
\caption{}
\label{fig3}
\end{figure}

Then we have that $f_s^{-1}(0) \cap \B_\theta \cong Cone(f_s^{-1}(0) \cap \BS_\theta)$ and $f_{s'}^{-1}(0) \cap \B_\theta \cong Cone(f_{s'}^{-1}(0) \cap \BS_\theta)$, for any $s' \in V_s$, and since we are supposing $\partial W_0$ is a fibre bundle, we have that $f_s^{-1}(0) \cap \BS_\theta \cong  f_{s'}^{-1}(0) \cap \BS_\theta$. Setting $\delta = \theta$, it follows that $f_s^{-1}(0) \cong  f_{s'}^{-1}(0)$, for any $s' \in V_s$. Since $L(\Sigma)$ is compact, we can choose $\delta >0$ sufficiently small so that $W_0$ is a fibre bundle.
\\
Now we just have to show the locally triviality of $p_|$. Given $s \in L(\Sigma)$, let $V_s$ be a neighborhood of $s$ in $L(\Sigma)$ such that $p$ is trivial in $V_s$, that is, there exists a homeomorphism $\Psi = (\Psi_1, \Psi_2):  p^{-1}(V_s) \to V_s \times \B_\delta(s)$. We have to show that $p_|^{-1}(V_s)$ is homeomorphic to $V_s \times p_|^{-1}(s)$, and then its enough to show that $p^{-1}(V_s) \cap L_0 \cong V_s \times \big(\B_\delta(s) \cap L_0 \big)$. This homeomorphism is given by the map $\varphi: p^{-1}(V_s) \cap L_0 \to V_s \times \big( \B_\delta(s) \cap L_0 \big)$ given by $\varphi(w):= (\Psi_1(w), h_{\Psi_1(w)}^{-1}(w))$, with $\varphi^{-1}(r_1,r_2):= h_{r_1}(r_2)$. Thus we get     that 
$(ii) \Rightarrow (i)$.

The proof of Proposition \ref {equiv} will be complete if we prove $(iii) \Rightarrow (iv)$; let us do so.
 By lemma \ref{lemma_Wt} we know that $\mathring{W_t}$ intersects $\B_\delta(s)$ transversally, for any $s \in L(\Sigma)$. Since $p_|: \partial W_t \to L(\Sigma)$ is the projection of a fibre bundle, it follows that it is a surjective submersion. Therefore $\partial W_t$ intersects $\B_\delta(s)$ transversally, for any $s \in L(\Sigma)$. Then it follows from the Ehresmann's fibration lemma \index{Ehresmann's fibration lemma} that $p_|: W_t \to L(\Sigma)$ is the projection of a fibre bundle.
\end{proof}

\begin{prop}
\label{theo_BWS}
The following conditions are equivalent:
\begin{itemize}
\item[$(i)$] $W_0$ is a fibre bundle over $L(\Sigma)$ with projection $p_0: W_0 \to L(\Sigma)$ and fibre $f_s^{-1}(0)$;
\item[$(ii)$] $L(\Sigma)$ is a stratum of a stratification of $W_0$ induced by a Whitney stratification of $f$;
\item[$(iii)$] Either $\Sigma$ or $\Sigma \backslash \{0\}$ is a stratum of a Whitney stratification of $f$;
\end{itemize}
\end{prop}

\begin{proof}
It is easy to see that $(ii)$ and $(iii)$ are equivalent. Let us  prove
$(i) \Rightarrow (ii)$.  Suppose that $f_s^{-1}(0) \cong f_{s'}^{-1}(0)$ for any $s,s' \in L(\Sigma)$. 
Given a point $p \in W_0$, let $\vec p$ denote the vector in $\bR^m$ defined by $p$. For any $p_1, p_2 \in W_0$, set $\overrightarrow{p_1p_2} = \vec p_1 + \vec p_2$ and denote $\overline{p_1p_2}$ the line in $\bR^m$ determined by $\overrightarrow{p_1p_2}$.
Let $(x_i)$ be a sequence of points in $Reg(W_0)$ and $(y_i)$ a sequence of points in $L(\Sigma)$ such that:
$$\displaystyle \lim_{i \to \infty} x_i = s \in L(\Sigma) ,\ \ \displaystyle \lim_{i \to \infty} y_i = s ,\ \ \displaystyle \lim_{i \to \infty} T_{x_i}(Reg W_0) = T \ \ \text{and} \ \displaystyle \lim_{i \to \infty} \overline{x_i y_i} = \lambda.$$
We have to show that $T \supset \lambda$.
Setting $s_i := p(x_i)$, we claim that:
 $$\displaystyle \lim_{i \to \infty} \overline{x_i y_i} \subseteq \displaystyle \lim_{i \to \infty} T_{x_i}(Reg W_0) \,,$$
 if and only if $\displaystyle \lim_{i \to \infty} \overline{x_i s_i} \subseteq \displaystyle \lim_{i \to \infty} T_{x_i}(f_{s_i}^{-1}(0))$. 
In fact, for any $s \in L(\Sigma)$, let $V_s$ be a neighborhood of $s$ in $L(\Sigma)$ such that $p_|^{-1}(V_s) \cong f_s^{-1}(0) \times V_s$. Then 
$$T_{x_i}(Reg W_0) = T_{x_i}(f_{s_i}^{-1}(0)) \times T_{s_i} (V_{s_i}) \,,$$
and therefore 
$$\displaystyle \lim_{i \to \infty} T_{x_i}(Reg W_0) = \displaystyle \lim_{i \to \infty} T_{x_i}(f_{s_i}^{-1}(0)) \times \displaystyle \lim_{i \to \infty} T_{s_i} (V_{s_i})\,.$$
Note that $\overrightarrow{x_i y_i} = \overrightarrow{x_i s_i} + \overrightarrow{s_i y_i}$ and therefore 
$$\displaystyle \lim_{i \to \infty} \overrightarrow{x_i y_i} = \displaystyle \lim_{i \to \infty} \overrightarrow{x_i s_i} + \displaystyle \lim_{i \to \infty} \overrightarrow{s_i y_i}.$$
Then 
$$\displaystyle \lim_{i \to \infty} \overline{x_i y_i} \subset \displaystyle \lim_{i \to \infty} T_{x_i}(Reg W_0)$$
if and only if 
$$\displaystyle \lim_{i \to \infty} \overrightarrow{x_i s_i} + \displaystyle \lim_{i \to \infty} \overrightarrow{s_i y_i} \subset \displaystyle \lim_{i \to \infty} T_{x_i}(f_{s_i}^{-1}(0)) \times \displaystyle \lim_{i \to \infty} T_{s_i} (V_{s_i}),$$
which happens if and only if 
$$\displaystyle \lim_{i \to \infty} \overline{x_i s_i} \subset \displaystyle \lim_{i \to \infty} T_{x_i}(f_{s_i}^{-1}(0)).$$

Now consider $h_i: \B_\delta(s_i) \to \B_\delta(s)$ the homeomorphism such that $h_i(f_{s_i}^{-1}(0)) = f_s^{-1}(0)$. We claim that:
$$\displaystyle \lim_{i \to \infty} \overline{x_i s_i} = \displaystyle \lim_{i \to \infty} \overline{(h_i(x_i))s}\,.$$
For this, set $R := \displaystyle \lim_{i \to \infty} \overline{x_i s_i}$. Then for each $p \in R$, there exists a sequence of points $(p_i)$ in $\bR^m$ such that $p_i \in  \overline{x_i s_i} $ and $p_i \to p$. 
Define $\tilde{p_i} := pr(h_i(p_i))$, where $pr$ is the orthogonal projection of $\B_\delta(s)$ into $\overline{(h_i(x_i))s}$. 
Since $p_i \to p$, it happens that for any arbitrarily small ball $\B(p)$ in $\bR^m$ centered at $p$ there exists a some $p_i$ in $\B(p)$, which implies that $h_i(p_i) \in \B(p) \cap \B_\delta(s)$, and then $\tilde{p_i} \in \B(p) \cap \B_\delta(s)$. Therefore $\tilde{p_i} \to p$.
Thus for all $p \in R$ there exists a sequence of points $(\tilde{p_i})$ in $\bR^m$ such that $\tilde{p_i} \in  \overline{(h_i(x_i))s}$ and $\tilde{p_i} \to p$, and hence $\displaystyle \lim_{i \to \infty} \overline{(h_i(x_i))s} = R$.

Now we claim $$\displaystyle \lim_{i \to \infty} T_{x_i}(f_{s_i}^{-1}(0)) = \displaystyle \lim_{i \to \infty} T_{h_i(x_i)}(f_s^{-1}(0))\,.$$ For this,    
set $S := \displaystyle \lim_{i \to \infty} T_{x_i}(f_{s_i}^{-1}(0))$. Then for each $p \in S$, there exists a sequence of points $(p_i)$ in $\bR^m$ such that $p_i \in  T_{x_i}(f_{s_i}^{-1}(0)) $ and $p_i \to p$. 
Consider $H_i: T_{x_i}(f_{s_i}^{-1}(0)) \to T_{h_i(x_i)}(f_s^{-1}(0))$ a homeomorphism between these two hyperplanes and define $\tilde{p_i} := H_i(p_i)$. 
Since $p_i \to p$, it happens that for any arbitrarily small ball $\B(p)$ in $\bR^m$ centered at $p$ there exists some $p_i$ in $\B(p)$, which implies that $H_i(p_i) \in \B(p) \cap \B_\delta(s)$, and therefore $\tilde{p_i} \to p$.  
Then $\forall p \in S$ there exists a sequence of points $(\tilde{p_i})$ in $\bR^m$ such that $\tilde{p_i} \in T_{h_i(x_i)}(f_s^{-1}(0))$ and $\tilde{p_i} \to p$, and hence $\displaystyle \lim_{i \to \infty}T_{h_i(x_i)}(f_s^{-1}(0)) = S$.               

So we have proved that 
$$\lambda \subseteq T \Leftrightarrow \displaystyle \lim_{i \to \infty} \overline{(h_i(x_i))s} \subseteq \displaystyle \lim_{i \to \infty} T_{h_i(x_i)}(f_s^{-1}(0)).$$
But since $f_s^{-1}(0)$ has an isolated singularity, it follows that $Reg(f_s^{-1}(0))$ and $ \{s\}$ are the two strata of a Whitney stratification of $f_s^{-1}(0)$. Hence $\displaystyle \lim_{i \to \infty} \overline{(h_i(x_i))s} \subseteq \displaystyle \lim_{i \to \infty} T_{h_i(x_i)}(f_s^{-1}(0))$, by the $(b)$-Whitney condition, and then we are done.

Finally, let us prove $(ii) \Rightarrow (i)$.  By the properties of Whitney stratifications, we know that for each $s \in L(\Sigma)$ there exists a neighborhood $U_s$ of $s$ in $W_0$ such that $U_s \cong V_s \times [f_s^{-1}(0) \cap U_s]$, where $V_s := U_s \cap L(\Sigma)$. Since $L(\Sigma)$ is compact, we can choose $\delta > 0$ sufficiently small such that $f_s^{-1}(0) \cap U_s = f_s^{-1}(0)$, for any $s \in L(\Sigma)$. Then we have that $p_|^{-1}(V_s) \cong V_s \times f_s^{-1}(0)$. In particular, $f_s^{-1}(0) \cong f_{s'}^{-1}(0)$, for any $s' \in V_s$.
\end{proof}

Combining  the results of this chapter we get:

\begin{theo} \label{theo_c4}
Let $f: (\bR^m,0) \to (\bR^k,0)$, with $m\geq k$, be a real analytic map-germ such that $0 \in \bR^k$ is an isolated critical value and $f$ has the Thom $a_f$-property. Then:

\begin{itemize}
\item[$(a)$] There exists a compact  regular neighbourhood $W$ of $L(\Sigma)$ in $\BS_\e$ with smooth boundary, such that for all $t$ with $\Vert t \Vert$ sufficiently small, one has that $L_t \backslash \mathring{W}$ is homeomorphic to $L_0 \backslash \mathring{W}$, where $ \mathring{W}$ is the interior of $W$. 

\item[$(b)$] If we suppose that $\Sigma$ has at most an isolated singularity,  then $W$  can be chosen to be a fibre bundle over $L(\Sigma)$ with fibre a disk.  

\item[$(c)$] In this latter case one has that the intersection $W_t:= L_t \cap W_t$ is a fibre bundle over $L(\Sigma)$ if and only if the intersection $W_0:= L_0 \cap W_t$ is a fibre bundle over $L(\Sigma)$, and this happens if and only if  $\Sigma \backslash \{0\}$ is a stratum of a Whitney stratification of $f$.
\end{itemize}
\end{theo}

\begin{cor} \label{cor_Wt}
In particular $W_t$ is a fibre bundle over $L(\Sigma)$ for the following map-germs:

\begin{itemize}
\item[$(a)$] $f: (\bC^n,0) \to (\bC,0)$ holomorphic with critical locus $\Sigma$ a complex curve (which was proved already by Siersma in \cite{Si});
\item[$(b)$] $\fgbar: (\bC^3,0) \to (\bC,0)$ where  $f,g: (\bC^3,0) \to (\bC,0)$ are holomorphic functions with no common irreducible components and such that $\fgbar$ has an isolated critical value. 
\end{itemize}
\end{cor}

\begin{proof} The first statement follows immediately from Theorem \ref{theo_c4}, since a germ of complex curve has at most an isolated singularity and hence $\Sigma \backslash \{0\}$ is a stratum of a Whitney stratification of $f$ (in fact, it is well-known that the singular set of a complex variety $V$ is an union of strata of a Whitney stratification $(S_\a)_{\a \in \Lambda}$ of $V$ with ${\rm Reg}(\Sigma) = S_\a$, for some $\a \in \Lambda$). The second statement follows by the same reason, since in this case  $\Sigma$ is the critical locus of $fg$, which is a complex curve.
\end{proof}

\vspace{.5cm}
\section{L\^e's Polyhedron for  real singularities}
\label{section_LP-real}

Consider now  a real analytic map germ 
$$f: (\bR^m,0) \to (\bR^k,0)\;, \quad 2 \leq k<m\,,$$
such that $0 \in \bR^m$ is an isolated critical point of $f$. Then $f$ has a local Milnor-L\^e fibration at $0$, that is, there exist $0 < \eta \ll \e$ sufficiently small such that the restriction
$$f_|: f^{-1}(\BD_\eta^*) \cap \B_\e \to \BD_\eta^*$$
is a locally trivial fibration.

In the definition below, just as in \cite{Le1}, by a polyhedron we mean a triangulable topological space.

\begin{defi} \label{defi_LP}
Let $F_f := f^{-1}(t) \cap B_\e$ be the local Milnor fibre of $f$ at $0$ and set $V:=  f^{-1}(0) \cap B_\e$. A L\^e polyhedron for the germ of $f$ at $0$ is a polyhedron $P \subset F_f$ such that:
\begin{enumerate}
\item $P$ is a deformation retract of $F_f$;
\item There is a continuous  map  $\Psi: F_f \to V$ such that $P = \Psi^{-1}(0)$ and \\ $\Psi: F_f \setminus P \to V \setminus \{0\}$ is a homeomorphism. Such a map is called  {\it collapsing map} for the local Milnor fibration.
\end{enumerate}
\end{defi}

Notice that if $P$ is such a polyhedron, then the homotopy type of $F_f$ equals that of $P$ and therefore $F_f$ can not have homology above the dimension of $P$. In \cite{Le1} L\^e Dung Trang described the degeneration of the Milnor fibre of  complex isolated hypersurface singularities  in terms of a vanishing polyhedron. In this section we extend L\^e's construction of a vanishing polyhedron to  real analytic map-germs of the type $\fgbar: (\bC^2,0) \to (\bC,0)$ with an isolated critical point at $0 \in \bC^2$.

\vspace{0.5cm}
\subsection{First step: constructing the polyhedron}

\ \\

Define the map $\phi: (\bC^2,0) \to (\bC^2,0)$ by setting $\phi(x,y) := \big(y, \fgbar(x,y)\big)$. Considering complex coordinates $x,\bar{x},y,\bar{y}$, the Jacobian matrix of $\phi$ is:

$$
J(\phi) =
\begin{pmatrix}
0 & 0 & 1 & 1 \\
0 & 0 & 1 & -1 \\
\frac{\partial f}{\partial x} \bar{g} & f \overline{\frac{\partial g}{\partial x}} & \frac{\partial f}{\partial y} \bar{g} & f \overline{\frac{\partial g}{\partial y}}  \\
\bar{f} \frac{\partial g}{\partial x} & \overline{\frac{\partial f}{\partial x}} g &
\bar{f} \frac{\partial g}{\partial y} & \overline{\frac{\partial f}{\partial y}} g &\\
\end{pmatrix} \;, 
$$
so the critical locus of $\phi$ is:
$$C(\phi) := \left\{ \big| f \frac{\partial g}{\partial x} \big|^2 - \big| g \frac{\partial f}{\partial x} \big|^2 =0 \right\}.$$
Notice that $C(\phi)$ is the zero-set of the real analytic function-germ $h: (\bR^4,0) \to (\bR,0)$ given by 
\begin{equation}\label{def h}
  h = \big| f \frac{\partial g}{\partial x} \big|^2 - \big| g \frac{\partial f}{\partial x} \big|^2 . \end{equation}

Geometrically, the set $C(\phi)$ consists of the critical points of $\fgbar$ together with the points $(x,y)$  which are regular points of $\fgbar$ and satisfy that  the kernel of the projection $\ell(x,y)=y$ is tangent to the fiber $(\fgbar)^{-1}\big(\fgbar(x,y)\big)$.

When $f$ is a holomorphic function-germ, one has that $C(\phi)$ and its image by $\phi$ are real surfaces in $\bR^4$ (actually they are complex curves), and this property is essential in the construction of a L\^e polyhedron by the methods of \cite{Le1}. But in the case of a real analytic map-germ of the type $\fgbar$ this is not always true. In fact, we have the following:

\begin{remark} {\rm
Given $\phi$ as above, let $h$ be as in equation {\rm (\ref{def h})}, so one has  that $C(h) \subseteq h^{-1}(0)$. We claim  that if there is a point in $h^{-1}(0)$ which is not   a critical point of $h$, then $\dim_\bR C(\phi) =3$. To see this, notice that
if there exists $z \in h^{-1}(0)$ such that $z$ is a regular point of $h$, then  $h^{-1}(h(x))$ has real dimension $3$ and since $h^{-1}(0) \supset h^{-1}(h(x))$ it follows $\dim_\bR C(\phi) = \dim_\bR h^{-1}(0) =3$.}
\end{remark}

Hence, in general 
we need some extra hypothesis on $\fgbar$ to ensure that $\phi\big( C(\phi) \big)$ is a real surface. The hypothesis we are going to ask from now on is that the map $g$  depends only on the variable $y$, {\it i.e.}, $g$ is of the form:
$$g(x,y) = g'(y)\,.$$
Then $\frac{\partial g}{\partial x} =0$ and we have that: 
$$C(\phi) = \{ g=0 \} \cup \left\{ \frac{\partial f}{\partial x} =0 \right\},$$
which is a complex curve in $\bC^2$.
Since $\{ g=0 \} \subset (\fgbar)^{-1}(0)$, it follows that the union of the irreducible components of $C$ that are not contained in $(\fgbar)^{-1}(0)$ is the complex curve
$$\Gamma = \left\{ \frac{\partial f}{\partial x} =0 \right\}.$$

Clearly, the same happens if one has that $g$ depends only on the variable $x$ or if the map $f$ depends only on one variable ($x$ or $y$).

Now set $\Delta := \phi(\Gamma)$, the polar image of $\fgbar$. Since $\phi$ is not holomorphic, $\Delta$ in general is not a complex curve in $\bC^2$, but it can be given a parametrization, as we show bellow.

Since $f$ and $g$ are holomorphic, we can write them as convergent power series 
$$f = \sum_{a,b=0}^\infty \kappa_{a,b} x^a y^b$$
and
$$g = \sum_{c=0}^\infty \kappa_c y^c,$$
with $\kappa_{a,b}, \kappa_c \in \bC$. Therefore the real analytic germ $\fgbar$ can be written as the convergent power series
$$\fgbar = \sum_{a,b,c=0}^\infty \kappa_{a,b,c} x^a y^b \bar{y}^c \,.$$
Now consider the decomposition of the complex curve $\Gamma$ into irreducible components:
$$\Gamma = \gamma_1 \cup \dots \cup \gamma_q.$$
Then for each $i = 1, \dots, q$ we can give $\gamma_i$ a Puiseux parametrization:

$$x = \sum_{j \geq 0} \a_{i,j} y^{m_{i,j}/n_i},$$
where $\a_{i,j} \in \bC$, with $\a_{i,0} \neq 0$, and $n_i,m_{i,j} \in \bN^*$. For each $i = 1, \dots, q$, set
$$\tau_i := \phi(\gamma_i) \subset \bC^2.$$
Then each $\tau_i$ admits a parametrization in the complex coordinates $(u,v)$ of $\bC^2$:

$$v(u) = \sum_{a,b,c=0}^\infty [\kappa_{a,b,c} u^b \bar{u}^c (\sum_{j \geq 0} \a_{i,j} u^{m_{i,j}/n_i})^a] .$$
In particular, if we consider the real analytic function
$$ 
\begin{array}{cccc}
\vartheta \ : & \! \bR^2 & \! \longrightarrow & \! \bR^2 \\
& \! u & \! \longmapsto & \! v(u)
\end{array} 
,$$
we get that $\tau_i \cap \{ v=t \} = \vartheta^{-1}(t)$, for any complex number $t$, which gives a finite union of points. Thus we arrive to the following lemma:

\begin{lemma}\label{polar curve}
Let $f, g$ be as above, with $g$ depending only on the variable $y$ and set $\phi = (y, \fgbar)$. Then:
\begin{enumerate}
\item The corresponding polar curve  is $C = \{ g=0 \} \cup \left\{ \frac{\partial f}{\partial x} =0 \right\},$ and therefore the union of the irreducible components of $C$ that are not contained in $(\fgbar)^{-1}(0)$ is the complex curve
$\Gamma = \left\{ \frac{\partial f}{\partial x} =0 \right\}.$
\item Each branch of the  real analytic variety $\Delta := \phi(\Gamma)$ has a Puisseux parametrization.
\item For each $t \in \bC$ with $|t| >0$ sufficiently small one has that each branch of $\Delta$ meets the plane $\bC \times \{\eta\} \subset \bC \times \bC$ in finitely many points.
\end{enumerate}
\end{lemma}

Clearly, we have an analogous lemma if $g$ depends only on the variable $x$ or if $f$ depends only on one variable ($x$ or $y$).

Now we can follow the same arguments of \cite{Le1} to construct a L\^e polyhedron for $f$ as follows. Consider small enough positive real numbers $\e, \eta_1, \eta_2$ with $0 < \eta_2 \ll \eta_1  \ll  \e  \ll 1$ such that $\phi$ induces a real analytic map
$$\phi_|: \B_\e \cap \phi^{-1}(\BD_{\eta_1} \times \BD_{\eta_2}) \to \BD_{\eta_1} \times \BD_{\eta_2}.$$
It restricts to a fibre bundle over $(\BD_{\eta_1} \times \BD_{\eta_2}) \backslash \Delta$, where $\Delta:= \phi_|(C)$ is the discriminant of $\phi_|$ in $\BD_{\eta_1} \times \BD_{\eta_2}$.

Notice that for any  $t \in \BD_{\eta_2} \backslash \{0\}$ the Milnor fibre $\B_\e \cap f^{-1}(t)$ can be identified with the real surface $F_t := \B_\e \cap \phi^{-1}(\BD_{\eta_1} \times \{t\})$. Fix $t \in \BD_{\eta_2} \backslash \{0\}$ and set
$$D_t := \BD_{\eta_1} \times \{t\}.$$
Hence $\phi_|$ induces a projection
$$\varphi_t: F_t \to D_t,$$
which is a finite covering over $D_t \backslash (\Delta \cap D_t)$. 

By the previous Lemma we know that the intersection $\Delta \cap D_t$ is a finite set of points, for any $t \in \BD_{\eta_2}$. Set:
$$\Delta \cap D_t  = \{y_1(t), \dots, y_k(t) \}.$$
Let $\lambda_t$ be a point in $D_t \setminus \{y_1(t), \dots, y_k(t)\}$ and for each $j = 1, \dots, k$, let $\delta(y_j(t))$ be a simple path (differentiable and with no double points) starting at $\lambda_t$ and ending at $y_j(t)$, such that two of them intersect only at $\lambda_t$. 
Set
$$Q_t := \bigcup_{j=1}^k \delta(y_j(t))$$
and
$$P_t := \varphi_t^{-1}(Q_t),$$
which is clearly a one-dimensional polyhedron in $F_t$.

\s
\begin{figure}[!h] 
\centering 
\includegraphics[scale=0.8]{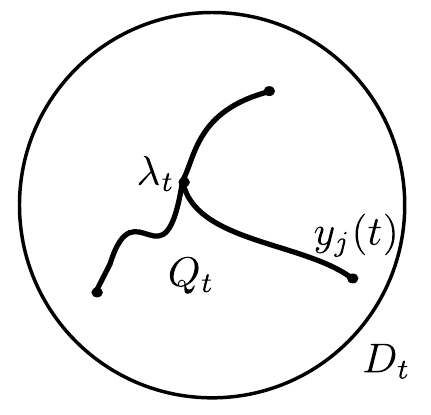}
\caption{}
\label{fig42_2}
\end{figure}
\s

\s
Now we prove that $P_t$ is a L\^e polyhedron for $f$ at $0$ as in Definition \ref{defi_LP}.

\begin{lemma} \label{lemma_v_t}
There exists a vector field $v_t$ in $D_t$ such that:  
\begin{enumerate}
\item It is differentiable (and hence integrable) over $D_t$;
\item It is null over $Q_t$;
\item It is  transversal to $\partial D_t$ and points inwards.
\item The associated flow $q_t: [0, \infty) \ \times \ (D_t \backslash Q_t) \to D_t$ defines a map
$$ 
\begin{array}{cccc}
\xi_t \ : & \! \partial D_t & \! \longrightarrow & \! Q_t \\
& \! u & \! \longmapsto & \! \displaystyle \lim_{\tau \to \infty} q_t(\tau,u)
\end{array} 
,$$
such that $\xi_t$ is continuous, surjective and differentiable.
\end{enumerate}
\end{lemma}

\begin{proof}
For each $j=1, \dots, k$, let $N_t^j$ be a small tubular neighbourhood of $\delta(y_j(t))$ in $D_t$ given by the normal bundle. That is, $N_t^j = \delta(y_j(t)) \times L_j$, where $L_j$ is a small line segment and each fibre $L_j' \stackrel {\rm homeo.}{\simeq} L_j$ contains the intersection point $p' = L_j' \cap \delta(y_j(t))$. On each fibre $L_j'$ consider the gradient vector field of the square of the function given by distance to the point $p'$. This defines a vector field $r_t^j$ on $N_t^j$ which is null over $\delta(y_j(t))$ and transversal to $\partial N_t^j$, pointing inwards. Now set
$$N_t:= \bigcup_{j=1}^k N_t^j,$$
which is a neighbourhood of $Q_t$ in $D_t$, and let $r_t$ be the sum of the vector fields $r_t^j$ by a partition of unity. Then $r_t$ is a $C^\infty$ vector field on $N_t$ which is null over $Q_t$, non-null over $N_t \backslash Q_t$, transversal to the boundary $\partial N_t$ and pointing inwards. Moreover, the associated flow $s_t: [0, \infty) \ \times \ (N_t \backslash Q_t) \to N_t$ defines a continuous, surjective and differentiable map
$$ 
\begin{array}{cccc}
\xi_t \ : & \! \partial N_t & \! \longrightarrow & \! Q_t \\
& \! u & \! \longmapsto & \! \displaystyle \lim_{\tau \to \infty} s_t(\tau,u)
\end{array} 
.$$
Then the vector field $v_t$ is obtained by gluing $r_t$, by a partition of unity,  with the restriction  to $D_t \backslash N_t$ of the gradient vector field of the square of the function given by the distance to the point $\lambda_t$.

\end{proof}

One has:

\begin{prop} We can choose a lifting of  $v_t$ to a vector field $\tilde{v}_t$ in $F_t$ so that: 
\begin{enumerate}
\item It is continuous over $F_t$;
\item  It is differentiable (and hence integrable) over $F_t \backslash P_t$;
\item  It is null over $P_t$;
\item  It is transversal to $\partial F_t$ and points inwards.
\item The associated flow $\tilde{q}_t: [0, \infty) \ \times \ (F_t \backslash P_t) \to F_t$ defines a map
$$ 
\begin{array}{cccc}
\tilde{\xi}_t \ : & \! \partial F_t & \! \longrightarrow & \! P_t \\
& \! z & \! \longmapsto & \! \displaystyle \lim_{\tau \to \infty} \tilde{q}_t(\tau,z)
\end{array} 
,$$
such that $\tilde{\xi}_t$ is continuous, surjective and differentiable. 
\item The fiber $F_t$ is homeomorphic to the mapping cylinder of $\tilde{\xi}_t$. 
\end{enumerate}
\end{prop}

\begin{proof}
Since $\varphi_t$ is a covering \index{covering} over $D_t \backslash Q_t$, which is differentiable (in the real sense), we can lift $v_t$ to a vector field $\tilde{v}_t$ in $F_t$ such that $\tilde{v}_t$  satisfies properties (1) to (4).  Let us show that we can choose $\tilde{v}_t$ satisfying also (5):

Fix $z \in \partial F_t$. We want to show that $\lim_{\tau \to \infty} \tilde{q}_t(\tau,z)$ exists, that is, that there exists a point $\tilde{p} \in P_t$ such that for any open neighbourhood $\tilde{U}$ of $\tilde{p}$ in $F_t$ there exists $\theta>0$ such that $\tau>\theta$ implies that $\tilde{q}_t(\tau,z) \in \tilde{U}$. 

From Lemma \ref{lemma_v_t} we know that there exists $p \in Q_t$ such that: 
$$\lim_{\tau \to \infty} q_t(\tau,\varphi_t(z)) = p\,.$$ Then we can fix an arbitrarily small open neighbourhood $U$ of $p$ in $D_t$ and guarantee that there exists $\theta>0$ such that $\tau>\theta$ implies that $q_t(\tau,\varphi_t(z)) \in U$.

Set $\{ \tilde{p}_1, \dots, \tilde{p}_r \} := \varphi_t^{-1}(p)$ and for each $j=1, \dots, r$ let $\tilde{U}_j$ be the connected component of $\varphi_t^{-1}(U)$ that contains $\tilde{p}_j$. Since $q_t(\tau,\varphi_t(z)) = \varphi_t(\tilde{q}_t(\tau,z))$ for any $\tau \geq 0$, we have that $\tau>\theta$ implies that $\varphi_t^{-1} \big( \varphi_t(\tilde{q}_t(\tau,z)) \big) \subset \varphi_t^{-1}(U)$. But then we clearly have that $\tilde{q}_t(\tau,z) \in \tilde{U}_j$ for some $j \in \{1, \dots, r\}$. Hence $\lim_{\tau \to \infty} \tilde{q}_t(\tau,z) = \tilde{p}_j$. This proves $(5)$.

Now we show that $F_t$ is homeomorphic to the mapping cylinder of $\tilde{\xi}_t$. In fact, the integration of the vector field $\tilde{v}_t$ gives a surjective continuous map 
$$\a: [0, \infty] \ \times \ \partial F_t \to F_t$$
that restricts to a diffeomorphism
$$\a_|: [0, \infty) \ \times \ \partial F_t \to F_t \backslash P_t.$$
Since the restriction $\a_{\infty}: \{\infty\} \times \partial F_t \to P_t$ is equal to $\tilde{\xi}_t$, which is differentiable and surjective, it follows that the induced map 
$$[\a_\infty]:  \big( (\{\infty\} \times \partial F_t ) \big/ \sim \big)  \to P_t$$
is a homeomorphism, where $\sim$ is the equivalent relation given by identifying $(\infty, z) \sim (\infty,z')$ if $\a_\infty(z) = \a_\infty(z')$.
Hence the map
$$[\a]:  \big( ([0,\infty] \times \partial F_t) \big/ \sim \big)  \to F_t$$
induced by $\a$ defines a homeomorphism between $F_t$ and the mapping cylinder of $\tilde{\xi}_t$.
\end{proof}

\begin{cor} 
The Fiber $F_t$ deformation retracts to the polyhedron $P_t$.
\end{cor}

\vspace{0.5cm}
\subsection{Second step: constructing the collapse map}

\ \\

We will do the construction of the vector field $\tilde{v}_t$ above simultaneously for all $t$ in a simple path $\gamma$ in $\BD_{\eta_2}$ joining $0$ and some $t_0 \in \partial \BD_{\eta_2}$, such that $\gamma$ is transverse to $\partial \BD_{\eta_2}$. The natural projection $\pi: \BD_{\eta_1} \times \BD_{\eta_2} \to  \BD_{\eta_2}$ restricted to $\Delta$ induces a ramified covering \index{covering}
$${\pi}_|: \Delta \to \BD_{\eta_2}$$
whose ramification locus is $\{0\} \subset \Delta$ (otherwise we would have Milnor fibres of $\fgbar$ with different homotopy type, which would be an absurd). 

Hence the inverse image of $\gamma \backslash \{0\}$ by this covering defines $k$ disjoint simple paths in $\Delta$, and each one of them is diffeomorphic to $\gamma \backslash \{0\}$. 

Recall that $\lambda_t$ is a point in $D_t \setminus \{y_1(t), \dots, y_k(t)\}$ and that $\delta(y_j(t))$ is a simple path starting at $\lambda_t$ and ending at $y_j(t)$, for each $j = 1, \dots, k$, such that two of them intersect only at $\lambda_t$. Then the set 
$$\Lambda := \bigcup_{t \in \gamma} \lambda_t$$
defines a simple path in $\BD_{\eta_1} \times \gamma$ such that $\Lambda \cap \Delta = \{0\}$. Moreover, we can choose the paths $\delta(y_j(t))$ in such a way that the set
$$T_j := \bigcup_{t \in \gamma} \delta(y_j(t)) \,,$$
forms a triangle and $T_j \backslash \{0\}$ is differentially immersed in 
$$\bigcup_{t \in \gamma} D_t = \BD_{\eta_1} \times \gamma \, ,$$
for each $j \in \{1, \dots, k\}$.
The intersection of any two triangles $T_j$ and $T_{j'}$ is the path $\Lambda$, for $j,j' \in \{1, \dots, k \}$ with $j \neq j'$. Set
$$Q := \bigcup_{j=1}^k T_j.$$

Now, let $\mathcal{V}$ be a vector field in $\BD_{\eta_1} \times \gamma$ such that $\mathcal{V}$ is:
\begin{itemize}
\item[$\bullet$] continuous;
\item[$\bullet$] null over $Q$;
\item[$\bullet$] differentiable over $(\BD_{\eta_1} \times \gamma) \backslash Q$;
\item[$\bullet$] transversal to $\partial \BD_{\eta_1} \times \gamma$; and such that
\item[$\bullet$] the projection of $\mathcal{V}$ on $\gamma$ is null.
\end{itemize}

Then the associated flow $w: [0, \infty) \ \times \ \big( (\BD_{\eta_1} \times \gamma) \backslash Q \big) \to \BD_{\eta_1} \times \gamma$ defines a map
$$ 
\begin{array}{cccc}
\xi \ : & \! \partial \BD_{\eta_1} \times \gamma & \! \longrightarrow & \! Q \\
& \! z & \! \longmapsto & \! \displaystyle \lim_{\tau \to \infty} w(\tau,z)
\end{array} 
,$$
such that $\xi$ is continuous, surjective and differentiable.

For any real $A>0$, set
$$V_A(Q):=  (\BD_{\eta_1} \times \gamma) \backslash w \big( [0,A) \times \partial \BD_{\eta_1} \times \gamma \big),$$
a closed neighbourhood of $Q$ in $\BD_{\eta_1} \times \gamma$. Note that $\partial V_A(Q)$ is a differentiable manifold that fibres over $\gamma$ with fibre a circle, by the restriction of the projection $\pi$. Moreover, $\BD_{\eta_1} \times \gamma$ is clearly the mapping cylinder of $\xi$. Now
set 
$$F_\gamma := \phi^{-1}(\BD_{\eta_1} \times \gamma) \cap \B_\e.$$
Since 
$$\phi_|: F_\gamma \backslash \phi^{-1}(Q) \to (\BD_{\eta_1} \times \gamma) \backslash Q$$
is a fibre bundle, it follows that $\phi^{-1}(\partial V_A(Q))$ is a differentiable submanifold of $F_\gamma$ which is a fibre bundle over $\gamma$. 
Now we set
$$P_\gamma := \phi^{-1}(Q),$$
which we call the {\it collapse cone of $\fgbar$ along $\gamma$}. Let $\theta$ be a vector field in $\gamma$ that goes from $t_0$ to $0$ in time $a>0$. 

Set
$Z := F_\gamma \backslash P_\gamma.$
Since 
$$Z = \phi^{-1} \big( (\BD_{\eta_1} \times \gamma) \backslash Q \big) \stackrel{\phi}{\longrightarrow} (\BD_{\eta_1} \times \gamma) \backslash Q \stackrel{\pi}{\longrightarrow} \gamma\,,$$
and
$$\phi^{-1} \big( \partial V_A(Q) \big) \stackrel{\phi}{\longrightarrow} \partial V_A(Q) \stackrel{\pi}{\longrightarrow} \gamma\,, $$
are differential fibre bundles, we can lift $\theta$ to obtain a vector field $\mathcal{E}$ such that:
\begin{itemize}
\item[$\bullet$] $\mathcal{E}$ is differentiable;
\item[$\bullet$] $\mathcal{E}$ is tangent to $\phi^{-1} \big( \partial V_A(Q) \big)$, for any $A > 0$.
\end{itemize}

Then the associated flow $h: [0,a] \times Z \to Z$ defines a $C^\infty$-diffeomorphism $\Psi$ from $F_{t_0} \backslash P_{t_0}$ to $F_0 \backslash \{0\}$ that extends to a continuous map from $F_{t_0}$ to $F_0$ and that sends $P_{t_0}$ to $\{0\}$. Hence we have proved:

\begin{cor} \label{theo_LPfgbar}
Let $f,g: (\bC^2,0) \to (\bC,0)$ be two holomorphic germs  such that the map-germ  $\fgbar: (\bC^2,0) \to (\bC,0)$ has an isolated critical point at $0 \in \bC^2$. Suppose that either $f$ or $g$ depends only on one variable. Then $\fgbar$ admits a one-dimensional L\^e polyhedron, that is, there exists a one-dimensional polyhedron $P$ in the local Milnor fibre $F_{\fgbar}$ of $\fgbar$, such that $F_{\fgbar}$ deformation retracts to $P$ and there exists a continuous map $\Psi: F_{\fgbar} \to V$, where $V$ is the local special fibre, such that $P = \Psi^{-1}(0)$ and $\Psi: F_{\fgbar} \setminus P \to V \setminus \{0\}$ is a homeomorphism.
\end{cor}

\vspace{.5cm}
\section{The degeneration of the boundary of the Milnor fibre of real analytic map germs of the type $\fgbar: (\bC^n,0) \to (\bC,0)$}
\label{section_degeneration}

In this section we describe the degeneration of the boundary of the Milnor fibre to the link of a real analytic map-germ $\fgbar: (\bC^n,0) \to (\bC,0)$, with $n \geq 3$, given by two reduced holomorphic germs $f,g: (\bC^n,0) \to (\bC,0)$ satisfying the following hypothesis:

\begin{itemize}
\item[$(A)$] $\fgbar$ has an isolated critical value at $0 \in \bC$;
\item[$(B)$] The critical set $\Sigma$ of $\fgbar$ satisfies that either $\Sigma$ or $\Sigma \backslash \{0\}$ is a stratum of a Whitney stratification of $(fg)^{-1}(0)$;
\item[$(C)$] Either $g$ is constant (and hence $\fgbar$ is holomorphic) or the function $f$ (or $g$) depends on only one variable and the map $(f,g): (\bC^n,0) \to (\bC^2,0)$ is a complete intersection, that is, the intersection $f^{-1}(0) \cap g^{-1}(0)$ has complex dimension $n-2$.
\end{itemize}

For example, every holomorphic
$$f: (\bC^n,0) \to (\bC,0)$$ 
 with critical locus $\Sigma$ a complex curve, satisfies these conditions (with $g\equiv 1$). 

\medskip

We have:

\begin{theo} \label{theo_degeneration}
Let $f,g: (\bC^n,0) \to (\bC,0)$ be two reduced holomorphic map-germs  such that the map-germ $\fgbar: (\bC^n,0) \to (\bC,0)$ satisfies the hypotheses $(A)$, $(B)$ and $(C)$. Then, for every $t\neq 0$ sufficiently small, there exist:
\begin{itemize}
\item[$(i)$] A small neighbourhood $W$ of the link of the critical locus $L(\Sigma)$ in the boundary of the Milnor ball $\BS_\e$ such that the part of $L_t$ (the boundary of the Milnor fibre of $\fgbar$) that is not inside $\mathring{W}$ is diffeomorphic to the part of $L_0$ (the link of $\fgbar$) that is not inside $\mathring{W}$.
\item[$(ii)$] A polyhedron $P_t$ in $W_t = L_t \cap W$, of real dimension $n+k-2$, such that $W_t$ deformation retracts to $P_t$. 
\item[$(iii)$] A continuous map $\Psi_t: W_t \to W_0 = L_0 \cap W$ which restricts to a homeomorphism from $W_t \backslash P_t$ to $W_0 \backslash L(\Sigma)$ and sends $P_t$ to $L(\Sigma)$.
\end{itemize}
\end{theo}

We already know from \cite{PS2} that the hypothesis $(A)$ implies  that there exist sufficient small positive reals $0 < \eta << \e$ such that the restriction 
$$\fgbar_|: (\fgbar)^{-1}(\BD_\eta^*) \cap \B_\e \to \BD_\eta^*\,,$$
is a fibre bundle (as we have already seen before). We also have: 

\begin{lemma} The hypothesis $(A)$ implies:
$$\Sigma = \Sigma(fg) = \Sigma(f) \cup \Sigma(g) \cup \big( f^{-1}(0) \cap g^{-1}(0) \big)\,,$$ so, in particular, $\Sigma$ is a complex variety.
\end{lemma}

\begin{proof}
In fact, if we consider coordinates $z_1, \bar{z_1}, \dots, z_n, \bar{z_n}$ for $\bC^n$, then
$$ \fgbar = (\Re \fgbar, \Im \fgbar) = \frac{1}{2} \biggl( \fgbar + \bar{f}g, \frac{1}{i} (\fgbar - \bar{f}g) \biggr) \,,$$
and the Jacobian matrix of the function $\fgbar$ is given  by:
$$
\begin{pmatrix}
\frac{\partial f}{\partial z_1} \bar{g} + \bar{f} \frac{\partial g}{\partial z_1} & f \bar{\frac{\partial g}{\partial z_1}} + g \bar{\frac{\partial f}{\partial z_1}} &  \dots & \frac{\partial f}{\partial z_n} \bar{g} + \bar{f} \frac{\partial g}{\partial z_n} & f \bar{\frac{\partial g}{\partial z_n}} + g \bar{\frac{\partial f}{\partial z_n}}  \\
\frac{\partial f}{\partial z_1} \bar{g} - \bar{f} \frac{\partial g}{\partial z_1} & f \bar{\frac{\partial g}{\partial z_1}} - g \bar{\frac{\partial f}{\partial z_1}} & \dots & \frac{\partial f}{\partial z_n} \bar{g} - \bar{f} \frac{\partial g}{\partial z_n} & f \bar{\frac{\partial g}{\partial z_n}} - g \bar{\frac{\partial f}{\partial z_n}}  
\end{pmatrix} \;.
$$
So the points in $\Sigma$ are determined by the following equations:
\begin{itemize}
\item[-] $fg (\frac{\partial f}{\partial z_i} \frac{\partial g}{\partial z_j} - \frac{\partial f}{\partial z_j} \frac{\partial g}{\partial z_i}) = 0$, for any $i \neq j$;
\item[-] $|f \frac{\partial g}{\partial z_i}| = |g \frac{\partial f}{\partial z_i}| $, for any $i= 0, \dots, n$;
\item[-] $|f|^2 \frac{\partial g}{\partial z_i} \bar{\frac{\partial g}{\partial z_j}} = |g|^2 \frac{\partial f}{\partial z_i} \bar{\frac{\partial f}{\partial z_j}}$, for any $i
\neq j$.
\end{itemize}
Since $\fgbar$ has an isolated critical value, then $\Sigma(\fgbar) \subset (fg)^{-1}(0)$. Define $\g = \big( f^{-1}(0) \cap g^{-1}(0) \big)$, a complex variety in $\bC^n$ of codimension $2$.

From the equations above we know that $\g \subseteq \Sigma$. If $x \in \Sigma(f)$, then $f(x)=0$ and $\frac{\partial f}{\partial z_i} =0$ for any $i=1, \dots, n$; hence $x \in \Sigma$ and therefore $\Sigma(f) \subseteq \Sigma$. In the same way we get $\Sigma(g) \subseteq \Sigma$. So $\Sigma \supseteq \g \cup \Sigma(f) \cup \Sigma(g)$. 

Now, if $x \in \Sigma$ and $x \notin \g$, let us assume  $f(x)=0$ and $g(x)\neq 0$. Then by the third equation above we see that $\frac{\partial f}{\partial z_i} = 0$, for any $i =1, \dots, n$, and hence $x \in \Sigma(f)$ (or if $f(x) \neq 0$ and $g(x)=0$, then $x \in \Sigma(g)$). Therefore $\Sigma = \g \cup \Sigma(f) \cup \Sigma(g)$.

Note that $\Sigma(fg) = \{ (\frac{\partial f}{\partial z_1}g +
f\frac{\partial g}{\partial z_1}, \dots, \frac{\partial f}{\partial z_n}g +
f\frac{\partial g}{\partial z_n}) = \underbar{0} \}$. Then
$\Sigma(f) \subset \Sigma(fg)$, $\Sigma(g) \subset \Sigma(fg)$ and
$\g \subset \Sigma(fg)$. So we get  that
$\Sigma \subseteq \Sigma(fg)$. 

On the other hand, if $x \in \Sigma(fg)$ and $x \notin \g$, then either $f(x)=0$ or $g(x)=0$. Suppose $f(x)=0$. Then $\frac{\partial f}{\partial
z_i}(x)= 0$ for any $i=1, \dots, n$, and hence $x \in \Sigma(f)$. In the same way, if $g(x)=0$, then $x \in \Sigma(g)$, so $\Sigma(fg) \subseteq
\Sigma$. Hence  $\Sigma = \Sigma(fg)$. 
\end{proof}

The next lemma is immediate from the previous discussion (see Theorem  \ref{theo_c4}):

\begin{lemma} The hypothesis $(B)$ implies that the critical locus $\Sigma$ of $\fgbar$ is either smooth or it is a complex analytic variety of dimension $k\leq n-2$ with an
 isolated singularity and one has:
 \begin{enumerate}
 \item The link  $L(\Sigma)$ is a smooth manifold and the vanishing zone $W$ is a fibre bundle over $L(\Sigma)$ with fibre a $(2n-2k)$-dimensional disk in $\BS_\e$. 
 \item For any $t \in \BD_\eta^*$, we can define the sets $F_t$, $F_0$, $L_t$, $L_0$, $W_t$ and $W_0$ as before, and   $W_t$ and $W_0$ are fibre bundles over $L(\Sigma)$.
 \end{enumerate}
\end{lemma}

Now let
$$\Sigma = \Sigma_1 \cup \dots \cup \Sigma_r$$ 
be the decomposition of $\Sigma$ into irreducible components. Then $L(\Sigma)$ has a corresponding decomposition into disjoint components:
$$L(\Sigma) =L(\Sigma_1) \sqcup \dots \sqcup L(\Sigma_r).$$
Hence $W$ has a decomposition into disjoint connected components given by
$$W = W[1] \sqcup \dots \sqcup W[r]$$
and $W_t$ and $W_0$ have decompositions into disjoint connected components given by
$$W_t = W_t[1] \sqcup \dots \sqcup W_t[r]$$
and
$$W_0 = W_0[1] \sqcup \dots \sqcup W_0[r].$$

Then each connected component $W_0[i]$ is a fibre bundle over the $(2k-1)$-manifold $L(\Sigma_i)$ and each connected component $W_t[i]$ is a fibre bundle over $L(\Sigma_i)$:

$$\xymatrix{ 
(\fgbar)_s^{-1}(t) \ \ar@{^{(}->}[r] & W_t[i] \ar[d]^{p} \\ 
	&	L(\Sigma_i) 
}
\hspace{1cm}
\xymatrix{ 
(\fgbar)_s^{-1}(0)  \ \ar@{^{(}->}[r] & W_0[i] \ar[d]^{p} \\ 
	&	L(\Sigma_i) 
}
$$
where the ball $\B_\theta(s)^{2n-2k} \subset \BS_\e^{2n-1}$ can be taken to be a Milnor ball for $(\fgbar)_s$ (see the claim of the proof of Theorem \ref{equiv}), and $(\fgbar)_s: \B_\theta(s) \to \bC$ has an isolated singularity at $0 \in \B_\theta(s)$. 

From now on we  consider the Milnor fibration of $\fgbar$ defined on a polydisk instead of on a ball. That is, we consider the polydisk 
$$\Delta_\e := \{ (z_1, \dots, z_n) \in \bC^n : max \{ |z_1|, \dots, |z_n| \} \leq \e \}\,,$$
instead of the ball $\B_\e$. Then we can choose a coordinate system $(z_1, \dots, z_n)$ of $\bC^n$ such that there exists $\e>0$ such that for any $\e' \leq \e$ the boundary of the polydisk $\Delta_\e$ intersects $\Sigma$ transversally at the open face
$$ \{ (z_1,\dots,z_n) \in \bC^n : max \{|z_1|, \dots, |z_{n-1}|\} < \e, |z_n| = \e \},$$
which contains $L(\Sigma)$. 

Then for each $s \in L(\Sigma)$, we can consider $(\fgbar)_s$ to be the real analytic isolated singularity function germ 
$$(\fgbar)_s: (H_s,s) \to (\bC,0)\,,$$
given by the restriction of $\fgbar$ to $H_s$, where $H_s$ is the $(n-k)$-dimensional affine hyperplane of $\bC^n$ passing through $s$ and parallel to $\{z_{n-k+1} = z_{n-k+2} \dots = z_n = 0\}$.

We remark  that $(i)$ of Theorem \ref{theo_degeneration} was proved already in Lemma \ref{lemma_vz}. We shall prove $(ii)$ and $(iii)$ of Theorem  \ref{theo_degeneration}. First we need to construct a L\^e Polyhedron for each isolated singularity $(\fgbar)_s$.

Notice that if $g$ is constant, then  $\fgbar$ is holomorphic, and the construction of a L\^e polyhedron for $(\fgbar)_s$ is given by L\^e's theorem in \cite{Le1}. So we 
assume that neither $g$ nor $f$ are constant.

By hypothesis $(C)$ we have that $\big( f^{-1}(0) \cap g^{-1}(0) \big)$ has complex dimension $(n-2)$. Since $\big( f^{-1}(0) \cap g^{-1}(0) \big)$ is contained in $\Sigma$, we must have $k \geq n-2$. But since $k \leq n-2$, it follows that $k=n-2$. Then $H_s$ is a $2$-dimensional affine hyperplane and the construction of a L\^e polyhedron for $(\fgbar)_s$ is given by Theorem \ref{theo_LPfgbar}.

Then we obtain: 

\begin{itemize}
\item[$\bullet$] A polyhedron $P_{t,s}$ of real dimension $n-k-1$ in the compact real surface $(\fgbar)_s^{-1}(t)$ such that $(\fgbar)_s^{-1}(t)$ deformation retracts to $P_{t,s}$. 
\item[$\bullet$] A continuous map $\Psi_{t,s}: (\fgbar)_s^{-1}(t) \to (\fgbar)_s^{-1}(0)$ which sends $P_{t,s}$ to $\{0\}$ and such that $\Psi_{t,s}$ restricts to a homeomorphism from $(\fgbar)_s^{-1}(t) \backslash P_{t,s}$ to $(\fgbar)_s^{-1}(0) \backslash \{0\}$.
\end{itemize}
Now fix $s_i \in L(\Sigma_i)$ and $t \in \BD_\eta^*$. Consider $P_{t,s_i}$ a L\^e Polyhedron for $(\fgbar)_{s_i}$ and $\Psi_{t,s_i}$ a collapsing map as in the previous lemma. Since $W_t[i]$ is a fibre bundle over $L(\Sigma_i)$, it follows that for any $s \in L(\Sigma_i) \backslash \{s_i\}$ there exists a homeomorphism 
$$h_{t,s}: \big( \B_\theta(s), (\fgbar)_{s}^{-1}(t) \big) \to \big( \B_\theta(s_i), (\fgbar)_{s_i}^{-1}(t) \big)$$
and a homeomorphism 
$$h_{0,s}: \big( \B_\theta(s), (\fgbar)_{s}^{-1}(0) \big) \to \big( \B_\theta(s_i), (\fgbar)_{s_i}^{-1}(0) \big),$$
where $\B_\theta(s_i) \subset H_{s_i}$ is a Milnor ball for $(\fgbar)_{s_i}$ (and therefore $\B_\theta(s) \subset H_s$ is a Milnor ball for $f_s$). Then for each $s \in L(\Sigma_i) \backslash \{s_i\}$ we define
$$P_{t,s} := h^{-1}(P_{t,s_i}).$$

\begin{lemma}
For each $s \in L(\Sigma_i) \backslash \{s_i\}$, the polyhedron $P_{t,s}$ defined as above is a L\^e Polyhedron for $(\fgbar)_s$. 
\end{lemma}

\begin{proof}
We  show first  that:
$(\fgbar)_s^{-1}(t)$ deformation retracts to $P_{t,s}$. 
Since $(\fgbar)_{s_i}^{-1}(t)$ deformation retracts to $P_{t,s_i}$, there exists a family of maps $\a_\kappa: (\fgbar)_{s_i}^{-1}(t) \to (\fgbar)_{s_i}^{-1}(t)$ in the real parameter $\kappa \in I$, where $I$ denotes the unit interval, such that $\a_0$ is the identity, $\a_1((\fgbar)_{s_i}^{-1}(t)) = P_{t,s_i}$ and the restriction of $\a_\kappa$ to $P_{t,s_i}$ is the identity, for any $\kappa \in I$.
Then for each $\kappa \in I$, define the map $\b_\kappa: (\fgbar)_s^{-1}(t) \to (\fgbar)_s^{-1}(t)$ by setting $\b_\kappa = h^{-1} \circ \a_\kappa \circ h$. This family of maps clearly defines a deformation retraction of $(\fgbar)_s^{-1}(t)$ onto $P_{t,s}$.

Now we prove that there exists a continuous map $\Psi_{t,s}: (\fgbar)_s^{-1}(t) \to (\fgbar)_s^{-1}(0)$ such that $\Psi_{t,s}$ restricts to a homeomorphism from $(\fgbar)_s^{-1}(t) \backslash P_{t,s}$ to $(\fgbar)_s^{-1}(0) \backslash \{s\}$ and such that $\Psi_{t,s}$ sends $P_{t,s}$ to $\{s\}$. 
Setting $\Psi_{t,s} := h_{0,s}^{-1} \circ \Psi_{t,s_i} \circ h_{t,s}$ we clearly obtain the desired map:
$$\xymatrix{ 
(\fgbar)_s^{-1}(t) \ar[r]^{h_{t,s}} \ar[d]_{\Psi_{t,s}} & (\fgbar)_{s_i}^{-1}(t) \ar[d]^{\Psi_{t,s_i}} \\ 
(\fgbar)_s^{-1}(0)	\ar[r]^{h_{0,s}} & (\fgbar)_{s_i}^{-1}(0)
}
$$

\end{proof}

We now set $P_t[i]$ to be the union of the $P_{t,s_i}$ with the union of all the $P_{t,s}$ for $s \in L(\Sigma_i) \backslash \{s_i\}$ defined as above, that is,
$$P_t[i] := \bigcup_{s \in L(\Sigma_i)} P_{t,s}.$$
Then clearly $P_t[i]$ is a fibre bundle over $L(\Sigma_i)$ with fibre $P_{t,s_i}$ and $W_t[i]$ deformation retracts to $P_t[i]$;
$$\xymatrix{ 
P_{t,s_i} \ \ar@{^{(}->}[r] & P_t[i] \ar[d]^{p} \\ 
	&	L(\Sigma_i) 
}
$$
Now define the continuous map $\Psi_t[i]: W_t[i] \to W_0[i]$ by setting 
$$\Psi_t[i](x) := \Psi_{t,p(x)}[i](x).$$
Clearly $\Psi_t[i]$ restricts to a homeomorphism from $W_t[i] \backslash P_t[i]$ to $W_0[i] \backslash L(\Sigma_i)$ and sends $P_t[i]$ to $L(\Sigma_i)$. Thus we arrive to Theorem \ref{theo_degeneration}.

\vskip.5cm


\end{document}